\documentclass[12pt]{amsart}

\usepackage[english]{babel}
\usepackage{amsthm}
\usepackage{amsmath}
\usepackage{amssymb}
\usepackage{amscd}
\usepackage[pdftex]{graphicx}
\usepackage[matrix,arrow,curve]{xy}
\usepackage{euscript}

\usepackage[pdftex]{hyperref}
\usepackage{lineno}

\theoremstyle{plain}

\newtheorem{theorem}{Theorem}
\newtheorem{lemma}[theorem]{Lemma}

\theoremstyle{definition}
\newtheorem*{definition}{Definition}

\newtheorem{example}[theorem]{Example}
\newtheorem{remark}[theorem]{Remark}

\renewcommand{\P}{{\mathbb P}}

\newcommand{\G}{\mathbb G}
\newcommand{\m}{\mathfrak m}

\newcommand{\spun}[1]{\left\langle {#1}\right\rangle}

\title{Additive structures on projective hypersurfaces}
\author{Ivan Bazhov} \thanks{The author was supported by Dynasty foundation, RFBR no. 12-01-31342 mol-a, and by The ministry of education and science of Russian federation (project 8214)}
\email{ibazhov@gmail.com}
\title{Additive structures on cubic hypersurfaces}
\date{}

\begin{document}
\begin{abstract}
By \emph{an additive structure on a hypersurface} $S$ in $\mathbb P^{n}$ we mean an effective 
action of commutative unipotent group on $\mathbb P^{n}$ which leaves $S$ invariant and acts on $S$ with an open orbit. It is known that these structures correspond to pairs $(R,H)$ of local algebra $R$ of dimension $n+1$ and a hyperplane $H$ in the maximal ideal $\m$ of $R$. We show when a projective hypersurface of degree $3$ has an additive structure and when structure is unique.
\end{abstract}
\maketitle{}


The previous works of the author was about toric varieties --- varieties with an open and dense orbit of an algebraic torus $T=(\G_m)^n$. Toric varieties contains two ideas: action of algebraic group and combinatorial description in terms of fans. The paper \cite{BBF} of the author is devoted to the second idea, this paper is of the first one, which has many different ways to develop (\cite{A1,A2,A7,A11,A18,A23}). 

The generalization that we have deal with consists of replacing the acting on a variety with an open orbit torus $T=(\G_m)^n$ by the group $\G=(\G_a)^n$, where $\G_m$ and $\G_a$ are the multiplicative and additive groups of the base field $\Bbbk$. Theory of $\G$-action with an open orbit can be regarded as an additive analogue of toric geometry. But there are certain differences. It is easy to show that $\G$-action with an open orbit on an affine variety is always transitive \cite[Section 1.3]{A20}. Hence we can not cover a variety with $\G$-invariant charts. It is well-known that any toric variety has a finite number of torus orbits and  and that an isomorphism between toric varieties as algebraic varieties provides their isomorphism in the category of toric varieties \cite[Th. 4.1]{A6}. There are examples when this do not hold in additive case \cite{A13}.

In \cite{A13} an algebraical description of locally-transitive $\G$-action on a projective space is given. Namely, it is shown there that such actions correspond to local finite-dimensional algebras. In \cite{AS} this correspondence is studied deeper: $\G$-actions on projective hypersurfaces are investigated, they correspond to pairs $(R,H)$, where $H\subset\m$ is a hypersurface in a maximal ideal $\m\subset R$.  Using \cite[Theorem 2.14]{A13} and ides of \cite{AS} and of \cite{A21} for hypersurfaces one can show that the following theorem holds.

\setcounter{theorem}{-1}
\begin{theorem}
\label{th_0}
There exists a correspondence between locally-transitive action of $\G^{n-1}_a$ on (normal)
hypersurfaces in $\mathbb P^{n}$ and isomorphisms classes of pairs $(R,H)$, where $R$ is a local algebra of dimension $n+1$ with the maximal ideal $\m$ and a hyperplane $H$ of $\m$, which elements generate algebra $R$.\qed
\end{theorem}

Investigation of $\G$-action on hypersurfaces is a complicated question. Description of $\G$-action with an open orbit on degenerated quadric was given in \cite{A21}: for every $n$ there is unique up to automorphism action. In the recent preprint \cite{AP} authors proved that additive structures on projective hypersurfaces of degree $d$ are in bijection with invariant $d$-linear symmetric forms on $(n+1)$-dimensional local algebras: the corresponding form is the polarization of the equation deﬁning the hypersurface.

We study the case of cubic projective hypersurfaces. Our goal is to find all cubics which has $\G$-action with an open orbit. We will show that if a cubic has such action than one can chooses coordinates such that the equation of the cubic has certain form. We also give an algorithm which allows to define if such form exists. We also show when $\G$-action with an open orbit is unique.

The author is grateful to I.\,Arzhantsev and to A.\,Popovsky for useful discussions, references and bright ideas.

\section*{Main result}

A given effective action of $\G=\mathbb G_a^{n-1}$ on $\mathbb P^{n}$ which leaves hypersurface $S$ invariant and acts on $S$ with an open orbit we will call an \emph{additive structure} on hypersurface $S$. We are going to find out which cubic $S$ admits an additive structure and when structure is unique.

\begin{definition}
A cubic form $Q(x_1,\ldots,x_k)$ in $k$ variables is called \emph{non-degenerated} if for any $\alpha_1,\ldots\alpha_k$ not always equal to 0 holds the following equation: 
$$
\alpha_1Q'_{x_1}(x_1,\ldots,x_k)+\ldots+\alpha_kQ'_{x_k}(x_1,\ldots,x_k)\neq0.
$$
\end{definition}

In other words a form $Q(x_1,\ldots,x_k)$ is non-degenerated if there is no linear transform of variables reducing the number of variables in $Q$.

\begin{theorem}
\label{Uno}
A cubic surface $Q=0$ in $\P^n$ admits an additive structure if and only if one can choose homogeneous variables 
$$
x_0,x_1,\dots,x_k,\ldots,x_s,y_0,y_1,\ldots,y_k,\ldots,y_{n-s-1}
$$ in $\P^n$, $(1\leq k\leq s\leq n-k-1)$, such that $Q$ has the following form
\begin{equation}
\label{eq_Q}
Q=x_0^2y_0+x_0(x_1y_1+\ldots+x_ky_k)+x_0(x_{k+1}^2+\ldots+x_s^2)+Q_3(x_1,\ldots,x_k),
\end{equation}
where $Q_3$ in a non-degenerated form in $k$ variables.

Moreover, an additive structure is unique if and only if $Q$ in non-degenerated in $n+1$ variables, i.e., $k+s+1=n$.
\end{theorem}

Remark \ref{r_form} gives ideas how one can $Q$ bring to form (\ref{eq_Q}) (if it is possible).

\begin{example}
Due to Remark \ref{r_form} and the proof of Lemma \ref{l_x0} any equation 
$$
x_0^3+Q(x_1,\ldots,x_n)=0, 
$$
where $Q'_{x_0}=0$, can not be represented in the form (\ref{eq_Q}). If the corresponding hypersurface is not a hyperplane then it does not have an additive structure.\hfill$\diamond$
\end{example}

\begin{example}
A surface given by the following equation provides a cubic with an additive structure.
$$
x_0(x_0y_0+x_1y_1)+x_1^3=0.
$$
If the dimension of the ambient space is $3$ then the $\mathbb G$-action is defined by
$$
\left(\begin{array}{r}
x_0\\
x_1\\
y_1\\
y_0
\end{array}\right)
\mapsto
\left(\begin{array}{r}
x_0\\
x_1+t_1x_0\\
y_1-2t_1x_1-2\left(t_2+\frac{t_1^2}{2}x_0\right)x_0\\
y_0-\frac12 t_1y_1+\left(t_2+\frac{t_1^2}{2}\right)x_1+\left(t_1t_2+\frac{t_1^3}{3}\right)x_0
\end{array}\right),
$$
ande the additive structure is defined by pair $(R,H)$:
$$
R=\Bbbk[a]/\spun{a^4},\ H=\spun{a,a^2},\ x_0^*=1,\ x_1^*=a,\ y_1^*=-2a^2,\ y_0^*=a^3. 
$$
In dimension $4$ there are two actions defined by algebras
$$
R=\Bbbk[a,b]/\spun{a^4,b^2},\ H=\spun{a,a^2, b},\ x_0^*=1,\ x_1^*=a,\ y_1^*=-2a^2,\ y_0^*=a^3,\ x_2^*=b; 
$$
$$
R=\Bbbk[a]/\spun{a^5},\ H=\spun{a,a^2, a^5},\ x_0^*=1,\ x_1^*=a,\ y_1^*=-2a^2,\ y_0^*=a^3, x_2^*=a^4. 
$$
\hfill$\diamond$
\end{example}

We will prove Theorem \ref{Uno} in $4$ steps:
\begin{enumerate}
\item existence of an additive structure implies existence of form (\ref{eq_Q});
\item existence of form (\ref{eq_Q}) implies existence of an additive structure;
\item if $Q$ is non-degenerated then additive structure is unique; 
\item if $k+s+1<n$ then there are at least two additive structures.
\end{enumerate}

We assume that the base field $\Bbbk$ is algebraically closed and of characteristic $0$.

\section*{Preliminaries}
Let $Q=0$ has an additive structure. Due to Theorem \ref{th_0}, additive structure is defined by a pair $(R,H)$, hence there is a decomposition 
\begin{equation}
\label{eq_R}
R=\spun{1}\oplus H\oplus \spun{b_0},
\end{equation}
where $b_0\in R\smallsetminus\spun{1}\oplus H$ is some vector. The proof of the following lemma is due to I.\,Arzhantsev and A.\,Popovsky, see \cite[Section 3]{AP}. 
\begin{lemma}
\label{polar}
If $Q=0$ of degree $3$ has an additive structure then
$$
\widetilde{Q}(a\alpha,\alpha',\alpha'')+\widetilde{Q}(\alpha,a\alpha',\alpha'')+\widetilde{Q}(\alpha,\alpha',a\alpha'')=0,
$$
for any $\alpha,\alpha',\alpha''\in R,\ a\in H$, and $\widetilde{Q}$ is a polarization of $Q$.\qed
\end{lemma}

\begin{lemma} 
\label{pr}
If hypersurace $Q=0$, $\mathrm{deg}\,Q=3$, has an additive structure given by a pair $(R,H)$ then the polarization $\widetilde{Q}$ is defined by 
$$
\widetilde Q(a,1,1)=Ay_0(a),
$$
$$
\widetilde Q(a,a',1)=-\frac12 Ay_0(aa'),
$$
$$
\widetilde Q(a,a',a'')=Ay_0(aa'a''),
$$
where $a,a',a''\in\m$, $A\in\Bbbk$ is a number, and $y_0$ is defined by: 
$$
y_0(b_0)=1,\ y_0(H\oplus\spun{1})=0. 
$$
\end{lemma}
\begin{proof}
Due to \cite[Theorem 5.1]{AS}, we can assume $b_0\in\m^3$. The space $H$ has a basis $a_1,\ldots,a_{n-1}$ such that $b_0=a_pa_q$ for some certain $a_p\in\m^2$ and $a_q\in\m$.
It is enough to check the statement of lemma for $a_1,\ldots,a_{n-1}$ and $b_0$.
\begin{enumerate}
\item 
\begin{enumerate}
\item for $a=b_0$ put $A=\widetilde{Q}(b_0,1,1)$
\item for $a\in H$ the statement follows from Lemma \ref{polar}
\end{enumerate}
\item 
\begin{enumerate}
\item for $a=a'=b_0$ using $y_0(b_0^2)=0$ we have
\begin{multline*}
\widetilde Q(1,a_pa_q,a_pa_q)=-\widetilde Q(a_p,a_q,a_pa_q)-\widetilde Q(1,a_q,a_p^2a_q)=\\
=\frac12\widetilde Q(a_p^2,a_q,a_q)+\frac12\widetilde Q(1,1,a_p^2a_q^2)=-\widetilde Q(1,a_p^2a_q,a_q)+\frac12\widetilde Q(1,1,a_p^2a_q^2)=\\
=\frac12\widetilde Q(1,1,a_p^2a_q^2)+\frac12\widetilde  Q(1,1,a_p^2a_q^2)=\widetilde Q(1,1,a_p^2a_q^2)=0
\end{multline*}
\item in other cases, without loss of generality, $a\in H$. Lemma \ref{polar} gives
$$
\widetilde Q(a,a',1)=-\frac12\widetilde Q(aa',1,1)=-\frac12 Ay_0(aa')
$$
\end{enumerate}
\item 
\begin{enumerate}
\item for $a=a'=a''=b_0$ we have
\begin{multline*}
\widetilde Q(a_pa_q,a_pa_q,a_pa_q)=-2\widetilde Q(a_q,a_p^2a_q,a_pa_q)=\widetilde Q(a_q,a_q,a_p^3a_q)=\\
=-2\widetilde Q(a_p^3a_q^2,a_q,1)=\widetilde Q(a_p^3a_q^3,1,1)=0
\end{multline*}
\item in other cases, without loss of generality, $a\in H$. Lemma \ref{polar} gives
$$
\widetilde Q(a,a',a'')=-\widetilde Q(aa',a'',1)-\widetilde Q(a',aa'',1)=Ay_0(aa'a'').
$$
\end{enumerate}
\end{enumerate}
\end{proof}

\begin{remark}
It is easy to see that the last lemma constructs the form $Q$ from $(R,H)$ up to a multiplicative constant:
\begin{equation}
\label{eq_RQ}
(R,H)\rightsquigarrow Q(R,H).
\end{equation} 
\end{remark}

\section*{Proof ot Theorem \ref{Uno}}
\begin{proof}[\bf First Step]
To show that for cubic $Q=0$ with an additive structure the form $Q$ can be written down as form (\ref{eq_Q})
we are going to choose a suitable basis in $R$. There is decomposition (\ref{eq_R}) and we can introduce two coordinates:
$$
x_0(1)=1,\ x_0(H\oplus\spun{b_0})=0, 
$$
$$
y_0(b_0)=1,\ y_0(H\oplus\spun{1})=0.
$$

Let us define two forms which are projections to $\spun{b_0}$ of a product in algebra:
$$\widetilde{Q}_{2}: H\times H\to \Bbbk,\ (a,b)\mapsto -\frac12 y_0(ab),$$
$$\widetilde{Q}_{3}: H\times H\times H\to \Bbbk,\ (a,b,c)\mapsto y_0(abc).$$
Let $Q_2$ and $Q_3$ be the corresponding quadratic and cubic forms on $H$. 

\begin{lemma}
\label{l_kers}
If $K_2$ is the kernel of $Q_2$ and $K_3$ is the kernel of $Q_3$ then
$$K_2\subset K_3.$$
\end{lemma}
\begin{proof}
There is a commutative diagram
$$
\xymatrix{
H\times(H\times H)\ar[rr]^{1\times\phi}\ar[rd]\ar[rdd]_{Q_3}&&H\times H\ar[ld]\ar[ldd]^{Q_2}\\
&\spun{b_0}\ar[d]&\\
&\Bbbk&
}
$$
where $\phi$ is a product in algebra and the projection on $H$ along $\spun{1,b_0}$. It is easy to see that $K_2\subset K_3$.
\end{proof}

\begin{lemma}
\label{l_basis_H}
One can choose a basis $x_1,\ldots, x_s$, $y_1,\ldots,y_r$ of space $H$ such that
$Q_{2}=x_1y_1+\ldots+x_ky_k+x_{k+1}^2+\ldots+x_{s}^2$ and $Q_{3}(x_1,\ldots,x_k)$ is non-degenerated.
\end{lemma}

\begin{proof}
There is decomposition 
$$
K_3=X\oplus Y\oplus K_2,
$$
where  $\ker(Q_2|_{K_3})=Y\oplus K_2$. The restriction of $Q_2$ to $X$ is non-degenerated and we can choose bases by the following way:
$$
X=\spun{a_{k+1},\ldots, a_{s}},\ \widetilde Q_{2}(a_i,a_j)=\delta_{ij},
$$
$$
Y=\spun{b_1,\ldots,b_k},\ K_2=\spun{b_{k+1},\ldots,b_r}.
$$
Since $b_i\notin K_2$ holds for $1\leq i\leq k$, there exists $a_i\in H\smallsetminus K_3$ such that $\widetilde Q_2(a_i,b_j)=\delta_{ij}$.

Assume that $\spun{a_1,\ldots,a_k}\oplus K_3\neq H$. In this case there exists
$a\in H\smallsetminus K_3$ such that $\widetilde Q_2(a, b)=0$ for all $b\in K_3$.
Since $a\notin K_3$, there exist two elements $b,c\in H$ such that $\widetilde Q_3(a,b,c)\neq 0$. Due to \cite[Theorem 5.1]{AS}, $y_0(\m^4)=0$, hence $y_0(ab\cdot c\cdot d)=0$ for all $c,d,\in H$. Hence $bc\in K_3$ and $\widetilde Q_2(a,bc)\neq 0$ gives a contradiction.

For decomposition
$$\spun{a_1,\ldots, a_k}\oplus\spun{b_1,\ldots,b_k}\oplus\spun{a_{k+1},\ldots,a_{s}}\oplus\spun{b_{k+1},\ldots,b_s}=H$$
and dual basis $1^*=x_0$, $x_i=a^*_i$, $y_i=b^*_i$ forms $Q_2$ and $Q_3$ are as in the statement of lemma.
\end{proof}

Due to Lemma \ref{pr}, cubic form $Q$ can be represent as
$$
Q=y_0x_0^2+x_0Q_{2}+Q_{3}.
$$
Applying Lemma \ref{l_basis_H} completes the proof of the first step.
\end{proof}

\begin{proof}[\bf  Second Step]

Basing on $Q$ of the form (\ref{eq_Q}) we construct a pair $(R, H)$ such that the correspondence (\ref{eq_RQ}) defines initial $Q$.

We have a set of variables $x_0,\ldots, x_s$, $y_0,\ldots,y_r$ (let us think that all variables not appearing in (\ref{eq_Q}) are labelled by $y_{k+1},\ldots,y_r$, $r=n-1-s$).
Set a basis:
$$R=\spun{1,a_1,\ldots,a_k, b_0,b_1,\ldots,b_r},$$
where $1^*=x_0$, $x_i=a^*_i$, $y_i=b^*_i$.
We have a space $H$ and forms $Q_2$ и $Q_3$:
\begin{equation}
\label{eq_QQ}
Q=y_0x_0^2+x_0Q_2+Q_3.
\end{equation}
The product of $a,a'\in H$ is given by
\begin{equation}
\label{eq_prod1}
aa'=\widetilde Q_2(a,a')b_0+\sum\limits_{q=1}^{k}\widetilde Q_3(a,a',a_q)b_q,
\end{equation}
and $b_0\cdot \m=0$.

It is easy to see that $Q$ coincides with a form from Lemma \ref{pr}, i.e., the correspondence defines the initial form.
\end{proof}


\begin{proof}[\bf Third Step.]
In the Second Step we have shown that if $Q$ has a form (\ref{eq_Q}) then there exists a pair $(R,H)$ giving an additive structure on $Q=0$.
Now we are going to show that if $Q$ is non-degenerated than this pair is unique. We will use Lemma \ref{pr}, which describes relation between ${Q}$ and product in algebra $R$.

\begin{lemma}
\label{l_x0}
If $Q$ is non-degenerated than decomposition (\ref{eq_R}) as a vector space is unique.
\end{lemma}
\begin{proof}
Direct computation shows that for non-degenerated form $Q$ there exist unique (up to multiplicative constant) $y_0$ such that $Q'_{y_0}$ is a perfect square. Hence (up to multiplicative constant) there is unique $x_0$: $x_0^2=Q'_{y_0}$.

Hyperplane $H\subset R$ can be defined as
$$
x_0(h)=y_0(h)=0.
$$
And  $\spun{1}=\spun{x_0^*}$, and $\spun{b_0}=\spun{y_0^*}$.
\end{proof}

We can put $1=x_0^*$ and $b_0=y_0^*$ in $R$ up to multiplicative constant, which has no affect. Decomposition (\ref{eq_QQ}) gives two forms $Q_2$ and $Q_3$. 

\begin{remark}
\label{r_form}
Now, it is easy to see how one can bring form $Q$ to (\ref{eq_Q}).
At first, one chooses variables $x_0$, $y_0$ and forms $Q_2$, $Q_3$. Secondly,
one brings $Q_2$ and $Q_3$ to a suitable form using ideas of Lemma \ref{l_basis_H}.
\end{remark}

\begin{lemma}
\label{lemma_2}
If form $Q$ in non-degenerated than multiplication table of $R$ is defined uniquely.
\end{lemma}
\begin{proof}
Due to the previous lemma, there is unique decomposition (\ref{eq_R}) and forms $Q_2$, $Q_3$.
Let us choose an arbitrary basis $a_i$ in $H$, $1\leq i\leq n-1$, and write down the product in terms of this basis:
\begin{equation}
\label{eq_prod2}
aa'=\mathrm{pr}_{b_0}(aa')b_0+\sum\limits_{i=1}^{n-1}\mathrm{pr}_{i}(aa')a_i.
\end{equation}

Due to Lemma \ref{pr}
\begin{itemize}
\item $\mathrm{pr}_{b_0}(aa')=-2\widetilde{Q_2}(a,a')$
\item since $\widetilde{Q_2}:H\times H\to\Bbbk$ is non-degenerated,  for any $a_i$ there exists unique 
$a_i^\vee\in H$ such that $\mathrm{pr}_i(aa')=-2\widetilde Q_2(aa',a_i^\vee)=\widetilde Q_3(a,a',a_i^\vee)$.
\end{itemize}
It is easy to define a multiplication for $a,a'\in H$:
$$
aa'=-2\widetilde Q_{2}(a,a')b_0+\sum\limits_{i=1}^{n}\widetilde Q_3(a,a',a_i^{\vee})a_i.
$$
Since $\widetilde Q_2$ in non-degenerated, $\m^3\smallsetminus\spun{b_0}=0$. Hence $H\cdot b_0=0$.
\end{proof}

The constructed multiplication of $R$ is coincide with multiplication from the Second Step.
\end{proof}

\begin{proof}[\bf Fourth Step.]
We are going to construct an algebra different from the algebra in the Second Step. But before it we recall (hidden) ideas from the previous step:
the second summand in (\ref{eq_prod1}) and (\ref{eq_prod2}) is defined by the diagram:
$$
\xymatrix{
H\times H\ar[rd]_{f_3}&&H \ar[ld]^{f_2}\\
&H^*&
}
$$
where $f_3$ and $f_2$ is defined by $Q_3$ and $Q_2$. In the case of (\ref{eq_prod1}) the inclusion  $\mathrm{Im}\,f_3\subset\mathrm{Im}\,f_2$ was important, and for (\ref{eq_prod2}) the map $f_2$ was an isomorphism, which helps to define a product. 

As in the second step we labelled missed in $\ref{eq_Q}$ variables as $y_{k+1},\ldots,y_r$, $r=n-1-s$, i.e., we have variables $x_0,\ldots, x_s$, $y_0,\ldots,y_r$ and basis:
$$R=\spun{1,a_1,\ldots,a_k, b_0,b_1,\ldots,b_r},$$
where $1^*=x_0$, $x_i=a^*_i$, $y_i=b^*_i$.
Also we have space $H$ and forms $Q_2$, $Q_3$. 

Let us write down what happens inside of diagram and change the maps:
$$
\xymatrix{
H\times H\ar[d]_{f_3}&H \ar[d]^{f_2}&&H\times H\ar[d]_{f_3+f_3'}&H \ar[d]^{f_2+f_2'}\\
\mathrm{Im}\,f_3\ar@{^{(}->}[r]&\mathrm{Im}\,f_2&&\mathrm{Im}\,f_3\oplus\mathrm{Ker}\,f_2\ar@{^{(}->}[r]&\mathrm{Im}\,f_2\oplus\mathrm{Ker}\,f_2
}
$$
\begin{itemize}
\item $f_3': H\times H\to \mathrm{Ker}\,f_2$ such that $\mathrm{Ker}\,f_3\neq \mathrm{Ker}\,(f_3+f_3')$,
\item $f_2'$ is identity on $\mathrm{Ker}\,f_2$ such that $f_2+f_2'$ is invertible.
\end{itemize}

Hence the multiplication of $R$ is given by
$$
ab=\widetilde Q_2(a,b)b_0+\sum\limits_{q=1}^{k} \widetilde Q_3(a,b,a_q)b_q+f'_3(a,b),
$$
and $b_0\cdot \m=0$.

Since the dimensions of $\m^2$ for new and old (from second step) algebras are different, we have different addititve structures.
\end{proof}

\end{document}